\documentclass[oneside,english]{amsart}
\usepackage[T1]{fontenc}
\usepackage[latin9]{inputenc}
\usepackage{enumitem}
\usepackage{amstext}
\usepackage{amsthm}
\usepackage{amssymb}

\makeatletter
\numberwithin{equation}{section}
\numberwithin{figure}{section}
  \theoremstyle{remark}
  \newtheorem*{acknowledgement*}{\protect\acknowledgementname}
\theoremstyle{plain}
\newtheorem{thm}{\protect\theoremname}[section]
  \theoremstyle{plain}
  \newtheorem{lem}[thm]{\protect\lemmaname}
  \theoremstyle{remark}
  \newtheorem{rem}[thm]{\protect\remarkname}
  \theoremstyle{plain}
  \newtheorem{prop}[thm]{\protect\propositionname}
  \theoremstyle{plain}
  \newtheorem{question}[thm]{\protect\questionname}

\usepackage{alex}

\theoremstyle{plain}

\usepackage{tikz-cd}
\usepackage{enumitem}

\makeatother

\usepackage{babel}
  \providecommand{\acknowledgementname}{Acknowledgement}
  \providecommand{\lemmaname}{Lemma}
  \providecommand{\propositionname}{Proposition}
  \providecommand{\questionname}{Question}
  \providecommand{\remarkname}{Remark}
\providecommand{\theoremname}{Theorem}

\begin{document}

\title[Representations of groups over rings of length two]{Representations of reductive groups over finite local rings of length
two}

\author{Alexander Stasinski and Andrea Vera-Gajardo}

\address{Alexander Stasinski, Department of Mathematical Sciences, Durham
University, Durham, DH1 3LE, UK}

\email{alexander.stasinski@durham.ac.uk}

\address{Andrea Vera-Gajardo, Instituto de Matemática, Universidad de Valparaiso,
Gran Bretaña 1091, Playa Ancha, Valparaíso.}

\email{andreaveragajardo@gmail.com}
\begin{abstract}
Let $\F_{q}$ be a finite field of characteristic $p$, and let $W_{2}(\F_{q})$
be the ring of Witt vectors of length two over $\F_{q}$. We prove
that for any reductive group scheme $\G$ over $\Z$ such that $p$
is very good for $\G\times\F_{q}$, the groups $\G(\F_{q}[t]/t^{2})$
and $\G(W_{2}(\F_{q}))$ have the same number of irreducible representations
of dimension $d$, for each $d$. Equivalently, there exists an isomorphism
of group algebras $\C[\G(\F_{q}[t]/t^{2})]\cong\C[\G(W_{2}(\F_{q}))]$.
\end{abstract}

\maketitle

\section{Introduction\label{sec:Introduction}}

Let $\cO$ be a discrete valuation ring with maximal ideal $\mfp$
and residue field $\F_{q}$ with $q$ elements and characteristic
$p$. For an integer $r\geq1$, we write $\cO_{r}=\cO/\mfp^{r}$.
Let $\cO'$ be a second discrete valuation ring with the same residue
field $\F_{q}$, and define $\cO_{r}'$ analogously. For a finite
group $G$, and an integer $d\geq1$, let $\Irr_{d}(G)$ denote the
set of isomorphism classes of irreducible complex representations
of $G$ of dimension $d$. It has been conjectured by Onn \cite[Conjecture~3.1]{Uri-rank-2}
(for $\lambda=r^{n}$, in the notation of \cite{Uri-rank-2}) that for all integers $r,n,d\geq1$,
we have
\[
\#\Irr_{d}(\GL_{n}(\cO_{r}))=\#\Irr_{d}(\GL_{n}(\cO_{r}')),
\]
or equivalently, that there exists an isomorphism of group algebras
$\C[\GL_{n}(\cO_{r})]\cong\C[\GL_{n}(\cO_{r}')]$. This was proved
for $r=2$ by Singla \cite{Pooja-GLn}. The conjecture makes sense
also when $\GL_{n}$ is replaced by any other group scheme $\G$ of
finite type over $\Z$, although in general small primes have to be
excluded. The analogous result was proved by Singla \cite{Pooja-classicalgrps}
for $r=2$ when $\G$ is either $\SL_{n}$ with $p\nmid n$ or an
adjoint form of a classical group of type $B_{n}$, $C_{n}$ or $D_{n}$,
provided that $p\neq2$. Regarding the case $\SL_{n}$ when $p\mid n$,
see Section~\ref{sec:Further-directions}.

In the present paper, we generalise Singla's results to arbitrary
reductive group schemes for which $p$ is a very good prime. More
precisely, we prove that for all $d\geq1$ and any reductive group
scheme $\G$ over $\Z$ such that $p$ is a very good prime for $\G\times\F_{q}:=\G\times_{\Spec\Z}\Spec\F_q$,
we have
\[
\#\Irr_{d}(\G(\cO_{2}))=\#\Irr_{d}(\G(\cO_{2}')).
\]
It is not hard to show that $\cO_{2}$, and indeed any commutative
local ring of length two with residue field $\F_{q}$, must be isomorphic
to one of the rings $\F_{q}[t]/t^{2}$ or $W_{2}(\F_{q})$ (see Lemma~\ref{lem:length-two-rings}).
From now on, let $R$ be either of these two rings.

Our main result, which we will now explain, is more general than the
above result in the sense that it covers a large class of representations
when $p$ is arbitrary and all representations when $p$ is very good.
As we explain in Section~\ref{subsec:Clifford-theory}, every irreducible
representation of $\G(R)$ determines a conjugacy orbit of one-dimensional
characters $\psi_{\beta}$ of the kernel of the canonical map $\rho:\G(R)\rightarrow\G(\F_{q})$,
and the characters $\psi_{\beta}$ are parametrised by elements $\beta$
in the $\F_{q}$-points of the dual Lie algebra $\Lie(\G\times\F_{q})^{*}$.
For any such $\beta$, let $\Irr_{d}(\G(R)\mid\psi_{\beta})$ denote
the set of irreducible representations of $\G(R)$ lying above $\psi_{\beta}$
and of dimension $d$. 

Let $k$ be an algebraic closure of $\F_{q}$. Let $G=(\G\times k)(k)$
(a reductive group), and let $\mfg^{*}$ be the dual of its Lie algebra.
We have a Frobenius endomorphism $F:G\rightarrow G$ corresponding
to the $\F_{q}$-structure on $G$ given by $\G\times\F_{q}$, and
we define a compatible endomorphism $F^{*}$ on $\mfg^{*}$
such that $(\mfg^{*})^{F^{*}}=\Lie(\G\times\F_{q})^{*}$ (see Section~\ref{subsec:The-Lie-algebra}).
The $p$th power map on $k$ gives rise to a bijection $\sigma^{*}:(\mfg^{*})^{F^{*}}\rightarrow(\mfg^{*})^{F^{*}}$,
which is related to Frobenius twists (see Sections \ref{subsec:Frobenius-twists} and \ref{subsec:Clifford-theory}).
Our main result is then:

\begin{thm}\label{thm:A} 
For any $\beta\in(\mfg^{*})^{F^{*}}$ such that $p$ does not divide the order of the component group $C_{G}(\beta)/C_{G}(\beta)^{\circ}$, and any $d\in\N$, we have 
\[ 
\#\Irr_{d}(\G(\F_{q}[t]/t^{2})\mid\psi_{\sigma^*(\beta)}))=\#\Irr_{d}(\G(W_{2}(\F_{q}))\mid\psi_{\beta}). 
\]
Moreover, if $p$ is good, not a torsion prime for $G$ and if there exists a $G$-equivariant bijection $\mfg\iso\mfg^{*}$, then the above condition on $p$ holds for all $\beta$, so that for any $d\in\N$, we have \[ \#\Irr_{d}(\G(\F_{q}[t]/t^{2}))=\#\Irr_{d}(\G(W_{2}(\F_{q}))). \]
\end{thm}The conditions on $p$ in the theorem hold for example when $p$ is
very good for $G$, or when $G=\GL_{n}$ (see Section~\ref{subsec:Very-good-primes}).
The proof builds on all the preceding results of the paper, and is
concluded in Section~\ref{subsec:Proof-of-the}. 

\subsubsection*{Method of proof and overview of the paper}

Our proof of Theorem~\ref{thm:A} is based on geometric properties
of the dual Lie algebra, together with results on centralisers in
algebraic groups, and group schemes over local rings. We give an outline
of the main steps of the proof, which may also serve as an overview
of the contents of the paper. We define a connected algebraic group
$G_{2}$ over $k$ with a surjective homomorphism $\rho:G_{2}\rightarrow G$
and a Frobenius map $F$ such that $G_{2}^{F}=\G(R)$. For each $\beta\in(\mfg^{*})^{F^{*}}$,
we want to show that $\#\Irr_{d}(\G(R)\mid\psi_{\beta})$ depends
only on structures over $\F_{q}$ (and not the choice of $R$ with
residue field $\F_{q}$). By elementary Clifford theory (see Lemma~\ref{lem:Clifford-theory}),
this will follow if there exists an extension of the character $\psi_{\beta}$
to its stabiliser 
\[
C_{G_{2}}(f)^{F}=C_{G_{2}^{F}}(f)
\]
in $G_{2}^{F}$, where $G_{2}^{F}$ acts on $(\mfg^{*})^{F^{*}}$
via its quotient $G^{F}$ and the coadjoint action, and $f=\beta$
or $(\sigma^{*})^{-1}(\beta)$ depending on whether $R=\F_{q}[t]/t^{2}$
or $W_{2}(\F_{q})$. Most of the paper is devoted to proving the existence
of an extension of $\psi_{\beta}$.

First, we use the known fact (Lemma~\ref{lem:p-Sylow-extend}) that
if $\psi_{\beta}$ extends to a Sylow $p$-subgroup of its stabiliser,
then it extends to the whole stabiliser. To show that $\psi_{\beta}$
extends to a Sylow $p$-subgroup, we work in the connected reductive
group $G$, as well as a connected algebraic group $G_{2}$ over $k$,
which we define using the Greenberg functor, and which is isomorphic
(as abstract group) to $\G(k[t]/t^{2})$ or $\G(W_{2}(k))$. We let
$G^{1}$ denote the kernel of $\rho$. Our key lemma (Lemma~\ref{lem:H_beta})
says that there exists a closed subgroup $H_{\beta}$ of $C_{G_{2}}(\beta)$
such that $H_{\beta}G^{1}$ is a maximal unipotent subgroup of $C_{G_{2}}(\beta)^{\circ}$,
$H_{\beta}\cap G^{1}=\exp(\Lie(U))$ and $\beta(\Lie(U))=0$. Here $U$ is the unipotent radical of a Borel subgroup of $G$ and
$\exp$ is a certain isomorphism between $\Lie(G)$ and $G^{1}$.
We show that if $p$ satisfies the conditions of the second part of Theorem~\ref{thm:A},
then $p$ does not divide the order of $C_{G}(\beta)/C_{G}(\beta)^{\circ}$,
and hence $H_{\beta}^{F^{n}}(G^{1})^{F^{n}}$ is a Sylow $p$-subgroup
of $G_{2}^{F^{n}}$ for any power $n$ such that $H_{\beta}$ is stable
under $F^{n}$ (see Lemmas~\ref{lem:p-not-div-comp-grp} and \ref{lem:ident-comp-p-Sylow}).
The properties $H_{\beta}\cap G^{1}=\exp(\Lie(U))$ and $\beta(\Lie(U))=0$
imply that $\psi_{\beta}$ is trivial on
$(H_{\beta}\cap G^{1})^{F^{n}}$. Thus, we can extend $\psi_{\beta}$
to $H_{\beta}^{F^{n}}(G^{1})^{F^{n}}$ by defining the extension to
be trivial on $H_{\beta}^{F^{n}}$. This implies that $\psi_{\beta}$
extends to its stabiliser $C_{G_{2}}(f)^{F^{n}}$ in $G_{2}^{F^{n}}$,
and restricting this extension (which is one-dimensional) to $C_{G_{2}}(f)^{F}$,
we finally obtain the sought-after extension.

In order to prove Lemma~\ref{lem:H_beta}, we need a couple of geometric
lemmas. First, we prove that the union of duals of Borel subalgebras
cover the dual Lie algebra (see Lemma~\ref{lem:conjugates-cover}).
This is an analogue of a theorem of Grothendieck that the union of
Borel subalgebras cover the Lie algebra, but our proof is analogous
to one by Borel and Springer. Next, we prove that any maximal connected
unipotent subgroup of $C_{G}(\beta)$ is contained in a maximal unipotent
subgroup of $G$ on whose Lie algebra $\beta$ is zero (see Lemma~\ref{lem:key-Borel-centr}).
This proof uses the preceding result on the union of dual Borel subalgebras
as well as Borel's fixed-point theorem.
\begin{acknowledgement*}
This work was carried out while the second author held a visiting
position at Durham University, supported by Becas Chile. The second
author was also supported by Conicyt through PAI/Concurso nacional
inserción de capital humano avanzado en la Academia convocatoria 2017
cod. 79170117.
\end{acknowledgement*}

\section{\label{sec:Alg-grps-Lie}Group schemes over local rings and $\F_{q}$-structures}

In this section, we will define the algebraic groups $G_{2}$ using
reductive group schemes over $R$ and the Greenberg functor. We will
also define $\F_{q}$-rational structures given by Frobenius endomorphisms
on $G_{2}$ as well as on the Lie algebra of $G$ and its dual. A
ring will mean a commutative ring with identity. We begin by characterising
local rings of length two:
\begin{lem}
\label{lem:length-two-rings}Let $A$ be a local ring of length two with
maximal ideal $\mfm$ and perfect residue field $F$. Then $A$ is
isomorphic to either $F[t]/t^{2}$ or $W_{2}(F)$. 
\end{lem}
\begin{proof}
The exact sequence 
\[
1\longrightarrow\mfm\longrightarrow A\longrightarrow F\longrightarrow1
\]
implies that the length of $\mfm$ is $1$ (since the length of $F$
is $1$). Thus $A$ cannot have any other proper non-zero ideals than
$\mfm$, so $\mfm$ is principal and $\mfm^{2}=0$. Note that $A$
is Artinian, hence complete.

If $\chara A=\chara F$, Cohen's structure theorem in equal characteristics
\cite[Theorem~9]{Cohen-structure-thm} implies that $A\cong F[[t]]/I$,
for some ideal $I$. Since every non-zero ideal of $F[[t]]$ is of
the form $(t^{i})$ and the length of $A$ is two, we must have $I=(t^{2})$.

If $\chara A\neq\chara F$, then $\chara F=p$ for some prime $p$,
and we have $p\in\mfm$. Note that $A$ is unramified in the sense
that $p\not\in\mfm^{2}$. Cohen's structure theorem in mixed characteristics
\cite[Theorem~12]{Cohen-structure-thm} implies that $A$ is a quotient
of an unramified complete discrete valuation ring $B$ of characteristic
$0$ and residue field $F$. By \cite[II, Theorems~3 and 8]{serre},
$B\cong W(F)$, the ring of Witt vectors over $F$ (this is where
the hypothesis that $F$ is perfect is used). Since the length of
$A$ is two, it must be the quotient of $W(F)$ by the square of the
maximal ideal, that is, $A\cong W_{2}(F)$.
\end{proof}
From now on, let $A$ be a finite local ring of length two. Then $A$
has finite residue field $\F_{q}$, for some power $q$ of a prime
$p$, and by the above lemma, $A$ is either $\F_{q}[t]/t^{2}$ or
$W_{2}(\F_{q})$.

Let $k=\bar{\F}_{q}$ be an algebraic closure of $\F_{q}$. All the
algebraic groups over $k$ which we will consider will be reduced,
and we will identify them with their $k$-points. In particular, although
centralisers are often non-reduced as group schemes, our notation
$C_{G}(\beta)$ will always refer to the $k$-points of the reduced
subgroup $C_{G}(\beta)_{\mathrm{red}}$. This makes our notation significantly
lighter, especially in Section~\ref{sec:Lemmas-on-algebraic}.

Let $\tilde{R}=k[t]/t^{2}$ if $R=\F_{q}[t]/t^{2}$ and $\tilde{R}=W_{2}(k)$
if $R=W_{2}(\F_{q})$, so that in either case $\tilde{R}$ has residue
field $k$. Let $\G$ be a reductive group scheme over $R$ (we follow
\cite[XIX, 2.7]{SGA3} in requiring that reductive group schemes have
geometrically connected fibres). 
For a scheme $X$ over $\Spec A$, where $A$ is a ring, and any ring homomorphism $A\rightarrow B$, we will (as in Section~\ref{sec:Introduction}) write 
$X\times B$, or $X\times_A B$, for $X\times_{\Spec A}\Spec B$.

Define the groups
\[
G_{2}=\cF_{\tilde{R}}(\G\times_{R}\tilde{R})(k)\qquad\text{\text{and}\qquad}G=(\G\times_{R}k)(k)=\G(k),
\]
where $\cF_{\tilde{R}}$ is the Greenberg functor with respect to
$\tilde{R}$ (see \cite{greenberg1}). Then $G_{2}$ and $G$ are
connected linear algebraic groups, and $G_{2}$ is canonically isomorphic
to $\G(\tilde{R})$, as abstract groups. 
\begin{rem}
The reason for the notation $G_{2}$ is that $\tilde{R}=\cO/\mfp^{2}$
for some complete discrete valuation ring $\cO$ with maximal ideal
$\mfp$, and $\G$ lifts to a reductive group scheme $\widehat{\G}$ over $\cO$, so $G_{2}$ sits in a tower of
groups $G_{r}=\cF_{\cO/\mfp^{r}}(\widehat{\G}\times_{\cO}\cO/\mfp^{r})(k)$,
for $r\geq2$. We will not need this.
\end{rem}

Let $\rho:G_{2}\rightarrow G$ be the surjective homomorphism induced
by the canonical map $\tilde{R}\rightarrow k$, and let $G^{1}$ denote
the kernel of $\rho$. By \cite[Section~5, Proposition~2 and Corollary~5]{greenberg1},
$\rho$ is induced on the $k$-points by a homomorphism
\[
\cF_{\tilde{R}}(\G)\longrightarrow\cF_{k}(\G\times k)=\G\times k
\]
of algebraic groups, so $G^{1}$ is closed
in $G_{2}$. By Greenberg's structure theorem \cite[Section~2]{greenberg2}, $G^1$ is abelian, connected and unipotent. We will freely use these properties of $G^1$ in the following.

\subsection{Frobenius endomorphisms\label{subsec:Frobenius-endomorphisms}}

Let $\phi$ be the unique ring automorphism of $\tilde{R}$ which
induces the Frobenius automorphism $\phi_{q}$ on the residue field
extension $k/\F_{q}$. In other words, $\phi$ is the map which raises
coefficients of elements in $k[t]/t^{2}$ (or coordinates of vectors
in $W_{2}(k)$) to the $q$-th power. Then the fixed points $\tilde{R}^{\phi}$
of $\phi$ is the ring $R$.

We have an $\F_{q}$-rational structure
\[
\cF_{R}(\G)\times_{\F_{q}}k\cong\cF_{\tilde{R}}(\G\times_{R}\tilde{R}),
\]
giving rise to a Frobenius endomorphism
\[
F:G_{2}\longrightarrow G_{2}
\]
such that
\[
G_{2}^{F}=\cF_{R}(\G)(\F_{q})\cong\G(R).
\]
Under some embedding of $\G\times\tilde{R}$ in an affine space $\A_{\tilde{R}}^{n}$,
the map $F$ is the restriction of the endomorphism on $\cF_{\tilde{R}}(\A_{\tilde{R}}^{n})(k)\cong\A_{\tilde{R}}^{2n}(\tilde{R})$
induced by $\phi$.

Similarly, the $\F_{q}$-rational structure $$(\G\times_{R}\F_{q})\times_{\F_{q}}k\cong\G\times_{R}\F_{q}$$
gives rise to a Frobenius endomorphism $F:G\rightarrow G$ (note that
we use the same notation as for the map on $G_{2}$) 
such that $$G^F=\G(\F_q)$$
and $F$ is the restriction of the $q$-th power map under the embedding of
$\G\times k$ in $\A_{k}^{n}$ corresponding to the embedding of $\G\times\tilde{R}$
in $\A_{\tilde{R}}^{n}$. Thus $\rho$ is compatible with the Frobenius maps on $G_2$ and $G$ in the sense that $\rho\circ F=F\circ\rho$, and it follows
from this that the kernel $G^{1}$ is $F$-stable. 

\subsection{The Lie algebra and its dual\label{subsec:The-Lie-algebra}}

Consider the reductive group $G$ over $k$ and let $\mfg=\Lie(G)$
be its Lie algebra. The $\F_{q}$-structure $G=(\G\times\F_{q})\times k$
gives rise to the $\F_{q}$-structure 
\[
\mfg=\Lie(\G\times\F_{q})\otimes k
\]
 on $\mfg$, and we denote the corresponding Frobenius endomorphism
by $F:\mfg\rightarrow\mfg$. The adjoint action
\[
\Ad:G\longrightarrow\GL(\mfg)
\]
comes from the adjoint action of $\G\times\F_{q}$ on $\Lie(\G\times\F_{q})$
by extension of scalars (see \cite[II, \S 4, 1.4]{Demazure-Gabriel-English}),
and thus it is compatible with the Frobenius maps in the sense that 

\begin{equation}
F(\Ad(g)X)=\Ad(F(g))F(X),\label{eq:Ad-F}
\end{equation}
for $g\in G$ and $X\in\mfg$.

Let $\mfg^{*}=\Hom_{k}(\mfg,k)$ be the linear dual of $\mfg$ and
let
\[
\langle{}\cdot{},{}\cdot{}\rangle:\mfg^{*}\times\mfg\longrightarrow k
\]
be the canonical pairing given by $\langle f,X\rangle=f(X)$. The
$k$-vector space structure on $\mfg^{*}$ gives rise to a structure
of affine space on $\mfg^{*}$, and we will consider $\mfg^{*}$ as
a variety with its Zariski topology. We have an endomorphism
\[
F^{*}:\mfg^{*}\longrightarrow\mfg^{*},\qquad f\longmapsto\phi_{q}\circ f\circ F^{-1}.
\]
This is compatible with the canonical pairing, in the sense that 
\[
\langle F^{*}(f),F(X)\rangle=(F^{*}(f))(F(X))=\phi_{q}\circ f\circ F^{-1}(F(X))=\phi_{q}\circ f(X)=\phi_{q}(\langle f,X\rangle).
\]
It follows from this that if $f\in(\mfg^{*})^{F^{*}}$ and $X\in\mfg^{F}$,
then $\phi_{q}(\langle f,X\rangle)=\langle f,X\rangle$, that is,
$\langle f,X\rangle\in\F_{q}$.

We will consider $\mfg^{*}$ with the coadjoint action of $G$, given
by $\Ad^{*}(g)f=f\circ\Ad^{-1}(g)$, for $g\in G$. The coadjoint
action is compatible with $F^{*}$, in the sense that
\begin{equation}
F^{*}(\Ad^{*}(g)f)=\Ad^{*}(F(g))F^{*}(f).\label{eq:coAd-F}
\end{equation}
Indeed, for $X\in\mfg$, we have
\[
F^{*}(\Ad^{*}(g)f)(X)=F^{*}(f\circ\Ad^{-1}(g))(X)=\phi_{q}\circ f\circ\Ad^{-1}(g)\circ F^{-1}(X)
\]
and on the other hand, by (\ref{eq:Ad-F}),
\begin{align*}
(\Ad^{*}(F(g))F^{*}(f))(X) & =F^{*}(f)\circ\Ad^{-1}(F(g))(X)=\phi_{q}\circ f\circ F^{-1}\circ\Ad^{-1}(F(g))(X)\\
 & =\phi_{q}\circ f\circ F^{-1}F(\Ad^{-1}(g)(F^{-1}(X)))\\
 & =\phi_{q}\circ f\circ\Ad^{-1}(g)\circ F^{-1}(X).
\end{align*}
 It follows from (\ref{eq:coAd-F}) that if $\beta\in(\mfg^{*})^{F^{*}}$,
then the centraliser $C_{G}(\beta)$ is $F$-stable.

\subsection{Frobenius twists and the kernel $G^{1}$ \label{subsec:Frobenius-twists}}

In order to describe the kernel $G^{1}$ and the conjugation action
of $G_{2}$ when $\tilde{R}=W_{2}(k)$, we need the notion of Frobenius
twists of schemes and representations.

Let $\sigma:k\rightarrow k$ be the homomorphism $\lambda\mapsto\lambda^{p}$,
and let $k_{\sigma}$ be the $k$-algebra structure on $k$ given
by $\sigma$. For any $k$-vector space $M$, its \emph{Frobenius
twist} is the base change
\[
M^{(p)}=M\otimes k_{\sigma}.
\]
For $m\otimes1\in M^{(p)}$, $m\in M$, and any $\lambda\in k$, we
thus have $\lambda(m\otimes1)=m\otimes\lambda=\lambda^{1/p}m\otimes1$.
In particular, if $X=\Spec A$, where $A$ is a $k$-algebra, we have
the Frobenius twist
\[
X^{(p)}=X\times k_{p}=\Spec A^{(p)},
\]
and the map $A^{(p)}\rightarrow A$, $a\otimes\lambda\mapsto\lambda a^{p}$
gives rise to a morphism $F_{X}:X\rightarrow X^{(p)}$. Any representation
$\alpha:G\rightarrow\GL(M)$ of $G$ as an algebraic group, induces
a representation $\alpha':G^{(p)}\rightarrow\GL(M^{(p)})$, which,
on the level of $k$-points has the effect
\[
\alpha'(g')(m\otimes1)=\alpha(g')m\otimes1,\qquad\text{for }g'\in G^{(p)}(k),m\in M.
\]
Composing $\alpha'$ with the map $F_{G}$, we get a representation
$\alpha^{(p)}$ of $G$ on $M^{(p)}$, which is the Frobenius twist
of $\alpha$. We have natural bijections
\[
X^{(p)}(k)=\Hom_{k}(A\otimes k_{\sigma},k)\cong\Hom_{k}(A,\Hom_{k}(k_{\sigma},k))\cong\Hom_{k}(A,k_{\sigma})=X(k_{\sigma}),
\]
where $X(k_{\sigma})$ coincides with the points obtained by applying
the map $\sigma:X(k)\rightarrow X(k)$ induced by $\sigma$. On $k$-points
we thus have, for $g\in G(k)$,
\begin{equation}
\alpha^{(p)}(g)\cdot(m\otimes1)=\alpha^{(p)}(\sigma(g))(m\otimes1)=\alpha(\sigma(g))m\otimes1.\label{eq:alpha-sigma}
\end{equation}
If the module $M$ is defined over $\F_{p}$, that is, if $M\cong M_{0}\otimes_{\F_{p}}k$,
for some $\F_{p}$-module $M_{0}$, then $M^{(p)}\cong(M_{0}\otimes k)\otimes k_{\sigma}\cong M_{0}\otimes k=M$
(note that $\sigma$ is $\F_{p}$-linear). In this situation, $\alpha^{(p)}$
is isomorphic to the representation $\alpha^{(p)}:G\rightarrow\GL(M)$
given by $\alpha^{(p)}(g)m=\alpha(\sigma(g))m$.

In terms of notation, let 
\[
\begin{cases}
\sigma^{(0)}=\Id & \text{if }\tilde{R}=k[t]/t^{2},\\
\sigma^{(p)}=\sigma & \text{if }\tilde{R}=W_{2}(k).
\end{cases}
\]

\begin{lem}
\label{lem:exp-compatibility}There exists an isomorphism of $k$-modules
$\exp:\mfg\rightarrow G^{1}$ such that, for $g\in G$, $X\in\mfg$,
we have
\[
g\exp(X)g^{-1}=\exp(\Ad(\sigma^{(i)}(g))X),\qquad i\in\{0,p\}.
\]
Moreover, $\exp\circ F=F\circ\exp$.
\end{lem}
\begin{proof}
When $\tilde{R}=k[t]/t^{2}$, the first statement follows from definition
of $\mfg$ as $G^{1}$, together with the corresponding definition
of $\Ad$ (see, \cite[II, \S 4, 1.2, 1.3 and 4.1]{Demazure-Gabriel-English}).
Assume now that $\tilde{R}=W_{2}(k)$. By \cite[A.6.2, A.6.3]{pseudo-reductive-book},
there exists an isomorphism $\exp^{(p)}:G^{1}\iso\mfg^{(p)}$ (equal
to $\theta_{1}^{-1}$ in loc.~cit.). The map $i(g):\G\rightarrow\G$,
$g\in\G(\tilde{R})$ defined by $h\mapsto ghg^{-1}$ is a homomorphism
of $\tilde{R}$-groups, so by \cite[A.6.2, A.6.3]{pseudo-reductive-book}
applied to $\bar{x}=1$, it induces the map $d(i(g))^{(p)}=\Ad(g)^{(p)}$
on $\mfg^{(p)}$. Thus, we have $g\exp^{(p)}(X)g^{-1}=\exp(\Ad^{(p)}(g)X)$,
for $X\in\mfg^{(p)}$. Since $\G\times\F_{p}$ provides an $\F_{p}$-structure
on $G$, we have an induced $\F_{p}$-structure on $\mfg$, so by
the above discussion of Frobenius twists, we have an isomorphism $\mfg^{(p)}\cong\mfg$,
which composed with $\exp^{(p)}$ gives the isomorphism $\exp$ satisfying
the asserted relation.

Finally, the relation $\exp\circ F=F\circ\exp$ follows in either
case by the description of $F$ on the points of $G_{2}$ and $\mfg$,
respectively.
\end{proof}

\section{\label{sec:Lemmas-on-algebraic}Lemmas on algebraic groups and Lie
algebra duals}

As before, $G$ will denote a connected reductive group over $k=\bar{\F}_{q}$.
Note however, that Lemmas~\ref{lem:conjugates-cover} and \ref{lem:key-Borel-centr}
hold for $G$ over an arbitrary algebraically closed field (including
characteristic $0$). 

Let $\Phi$ be the set of roots with respect to a fixed maximal torus
$T$ of $G$. Let $B$ be a Borel subgroup of $G$ containing $T$,
determining a set of positive roots $\Phi^{+}$. Following Kac and
Weisfeiler \cite{Kac-Weisfeiler} (who attribute this to Springer;
see \cite[Section~2]{Springer-trigonometric-sums}), we define
\[
\mfb^{*}=\{f\in\mfg^{*}\mid f(\Lie(U))=0\},
\]
where $U$ is the unipotent radical of $B$.
Since $\mfb^{*}$ is a linear subspace of $\mfg^{*}$, it is closed.

By a well known result of Borel, $G$ is the union of its Borel subgroups,
and an analogous theorem of Grothendieck says that $\mfg$ is the
union of its Borel subalgebras (i.e., Lie algebras of Borel subgroups);
see \cite[14.25]{borel} or \cite[XIV~4.11]{SGA3}. In \cite[Lemma~3.3]{Kac-Weisfeiler}
the analogous statement for the dual $\mfg^{*}$ is claimed under
the hypotheses that $p\neq2$ and $G\neq\mathrm{SO}(2n+1)$. Since
the argument in \cite{Kac-Weisfeiler} is short on details and omits
non-trivial steps (such as the existence of regular semisimple elements
in $\mfg^{*}$ when $p\neq2$), we give a complete proof for any $p$
and a reductive group $G$. Note that while Borel's and Grothendieck's theorems
hold for any connected linear algebraic group, the dual Lie algebra
version does not. For example, for a unipotent group, $\mfb^{*}$
defined as above, would just be $0$.

For each $\alpha\in\Phi$, let 
\[
x_{\alpha}:k\longrightarrow U_{\alpha}\subset G
\]
 be the corresponding isomorphism such that $tx_{\alpha}(u)t^{-1}=x_{\alpha}(\alpha(t)u)$
for all $t\in T$, $u\in k$.

Let $X_{\alpha}:k\rightarrow\mfg$ denote the differential of $x_{\alpha}$,
so that $\Ad(t)X_{\alpha}(u)=X_{\alpha}(\alpha(t)u)$ for all $t\in T$,
$u\in k$. We write
\[
e_{\alpha}:=x_{\alpha}(1),\qquad E_{\alpha}:=X_{\alpha}(1).
\]
Furthermore, we define $E_{\alpha}^{*}\in\mfg^{*}$ via
\[
\begin{cases}
\langle E_{\alpha}^{*},E_{-\alpha}\rangle=1,\\
\langle E_{\alpha}^{*},E_{\beta}\rangle=0 & \text{if }\beta\neq-\alpha,\\
\langle E_{\alpha}^{*},\Lie(T)\rangle=0.
\end{cases}
\]
It is well known that the Weyl group $W$ of $G$ with respect to
$T$ acts on the elements $E_{\alpha}$ by $w(E_{\alpha}):=\Ad(\dot{w})E_{\alpha}=E_{w(\alpha)}$,
where $\dot{w}\in N_{G}(T)$ is a representative of $w\in W$. It
also acts on the elements $E_{\alpha}^{*}$ by $w(E_{\alpha}^{*}):=\Ad^{*}(\dot{w})E_{\alpha}^{*}$.
We thus have 
\[
w(E_{\alpha}^{*})(E_{\beta})=E_{\alpha}^{*}(\Ad(\dot{w})^{-1}E_{\beta})=E_{\alpha}^{*}(E_{w^{-1}(\beta)})=\begin{cases}
1 & \text{if }-\alpha=w^{-1}(\beta),\\
0 & \text{otherwise},
\end{cases}
\]
and moreover $w(E_{\alpha}^{*})(\Lie(T))=0$ because $w$ preserves
$\Lie(T)$. Therefore, since $-\alpha=w^{-1}(\beta)$ is equivalent
to $\beta=-w(\alpha)$, we have
\[
w(E_{\alpha}^{*})=E_{w(\alpha)}^{*}.
\]
The proof of the following result follows the same lines as \cite[14.23-14.24]{borel}.
We give the proof here for the sake of completeness.
\begin{lem}
\label{lem:conjugates-cover}The conjugates of ${\bf \mfb^{*}}$ cover
$\mfg^{*}$, that is, $\mfg^{*}=\bigcup_{g\in G}\Ad^{*}(g)\mfb^{*}$.
\end{lem}
\begin{proof}
First, we prove that there exists an $n\in\mfb^{*}$ such that 
\begin{equation}
\{g\in G\mid\Ad^{*}(g)n\in\mfb^{*}\}=B.\label{eq:TrXb-is-B}
\end{equation}
 Let $\Delta\subset\Phi^{+}$ be a set of simple roots and define
\[
n=\sum_{\alpha\in\Delta}E_{\alpha}^{*}.
\]
Let $g\in\{g\in G\mid\Ad^{*}(g)n\in\mfb^{*}\}$. By the Bruhat decomposition
of $G$, we may write $g=b'\dot{w}b$, for $b,b'\in B$, $\dot{w}\in N_{G}(T)$. Since $B$ normalises $\mfb^{*}$,
we may assume $b'=1$. The formula
\[
\Ad^{*}(x_{\alpha}(u))E_{\beta}^{*}=E_{\beta}^{*}+\sum_{i\geq1}c_{i}u^{i}E_{\beta+i\alpha}^{*},\qquad\text{for all }u\in k,\ \alpha,\beta\in\Phi,\ \beta\neq-\alpha,
\]
where $c_{i}\in k$ (see the third equation in \cite[2.2, (1)]{Springer-trigonometric-sums};
note the two missing dashes) implies that 
\[
\Ad^{*}(b)n-n\in\langle E_{\alpha}^{*}\mid\alpha\in\Phi^{+},\,\alpha\not\in\Delta\rangle.
\]
Note that the condition $\alpha\notin\Delta$ above is due to the
fact that if $\beta\in\Delta$ and $\alpha\in\Phi^{+}$, then $\beta+i\alpha\not\in\Delta$,
for any $i\geq1$ (since $\Delta$ is linearly independent). Thus
\[
\Ad^{*}(g)n=\Ad^{*}(\dot{w})\Big(n+\sum_{\alpha\in\Phi^{+}\setminus\Delta}c_{\alpha}E_{\alpha}^{*}\Big),
\]
for some $c_{\alpha}\in k$.

Since $w$ permutes the $E_{\alpha}^{*}$ according to $w(E_{\alpha}^{*})=E_{w(\alpha)}^{*}$,
we conclude that
\[
\Ad^{*}(g)n=\sum_{\alpha\in\Delta}E_{w(\alpha)}^{*}+\sum_{\alpha\in\Phi^{+}\setminus\Delta}c_{\alpha}E_{w(\alpha)}^{*}.
\]
Since the sets $\{w(\alpha)\mid\alpha\in\Delta\}$ and $\{w(\alpha)\mid\alpha\in\Phi^{+}\setminus\Delta\}$
are disjoint and the $E_{\alpha}^{*}$ are linearly independent, the
condition $\Ad^{*}(g)\in\mfb^{*}$ implies that $w(\Delta)\subseteq\Phi^{+}$.
As is well known, this implies that $w=1$; hence $g\in B$.

Next, consider the morphisms
\[
G\times\mfg^{*}\xrightarrow{\ \phi_{1}\ }G\times\mfg^{*}\xrightarrow{\ \phi_{2}\ }G/B\times\mfg^{*},
\]
where $\phi_{1}(g,f)=(g,\Ad^{*}(g)f)$ and $\phi_{2}(g,f)=(gB,f)$,
for any $f\in\mfg^{*}$. Let 
\[
M=\phi_{2}\phi_{1}(G\times\mfb^{*})=\{(gB,f)\mid g\in G,\,\Ad^{*}(g)^{-1}f\in\mfb^{*}\}.
\]
The fibre over $gB$ of the surjective projection $\mathrm{pr}_{1}:M\rightarrow G/B$
is isomorphic to $\Ad^{*}(g)\mfb^{*}$, so the dimension of each fibre
is $\dim\mfb^{*}=\dim B$. Hence $\dim M=\dim G/B+\dim B=\dim G$.
On the other hand, the fibre $\mathrm{pr}_{2}^{-1}(l)$ of the second
projection $\mathrm{pr}_{2}:M\rightarrow\mfg^{*}$ over any $l\in\mfg^{*}$
is isomorphic to
\[
\{gB\mid\Ad^{*}(g)^{-1}l\in\mfb^{*}\}=\{g\mid\Ad^{*}(g)^{-1}l\in\mfb^{*}\}/B.
\]
It follows from (\ref{eq:TrXb-is-B}) that $\mathrm{pr}_{2}^{-1}(n)=\{1\}$,
so in particular, there exist finite non-empty fibres of $\mathrm{pr}_{2}$
in $M$. Therefore, since $M$ is irreducible (being the image of
the irreducible set $G\times\mfb^{*}$), the fibres of $\mathrm{pr}_{2}:M\rightarrow\mfb^{*}$
are finite over some dense open set in $\overline{\mathrm{pr}_{2}(M)}$.
Since $\dim M=\dim G=\dim\mfg^{*}$ and $\mfg^{*}$ is connected,
it follows that $\mathrm{pr}_{2}:M\rightarrow\mfg^{*}$ is dominant.
Thus $\mathrm{pr}_{2}(M)=\bigcup_{g\in G}\Ad^{*}(g)\mfb^{*}$ contains
a dense subset of $\mfg^{*}$.

We show that $M$ is closed in $G/B\times\mfg^{*}$. If $\Ad^{*}(g)^{-1}f\in\mfb^{*}$,
then $\Ad^{*}(gb)^{-1}f\in\mfb^{*}$, for all $b\in B$, so $\phi_{2}^{-1}(M)=\phi_{1}(G\times\mfb^{*})$.
Thus, since $\phi_{1}$ is an isomorphism of varieties, $\phi_{2}^{-1}(M)$
is closed. Since $\phi_{2}:G\times\mfg^{*}\rightarrow(G\times\mfg^{*})/(B\times\{0\})$
is a quotient morphism (hence open), the set
\[
\phi_{2}(G\times\mfg^{*}\setminus\phi_{2}^{-1}(M))=(G/B\times\mfg^{*})\setminus M
\]
is open, so $M$ is closed.

Finally, since $G/B$ is a complete variety, the image of $M$ under
the projection $\mathrm{pr}_{2}:G/B\times\mfg^{*}\rightarrow\mfg^{*}$
is closed. But $\mathrm{pr}_{2}(M)=\bigcup_{g\in G}\Ad^{*}(g)\mfb^{*}$,
which we have shown is dense in $\mfg^{*}$. Thus $\bigcup_{g\in G}\Ad^{*}(g)\mfb^{*}$
is closed and dense, so $\mfg^{*}=\bigcup_{g\in G}\Ad^{*}(g)\mfb^{*}$.
\end{proof}
In the following lemma, the proof follows the lines of the first part
of the proof of ii) on p.~143 of \cite{Kac-Weisfeiler}, but in addition,
we also provide a proof of the fact that $X$ is closed in $G/B$.
\begin{lem}
\label{lem:key-Borel-centr}Let $\beta\in\mfg^{*}$, $B_{\beta}$
be a Borel subgroup of $C_{G}(\beta)$ and $U_{\beta}$ be the unipotent
radical of $B_{\beta}$. Then there exists a Borel subgroup of $G$
with unipotent radical $V$ such that 
\[
U_{\beta}\subseteq V\qquad\text{and}\qquad\beta(\Lie(V))=0.
\]
\end{lem}
\begin{proof}
Let $B$ be a fixed Borel subgroup of $G$ with unipotent radical
$U$, and define the set
\[
X=\{gB\in G/B\mid\beta(\Ad(g)\Lie(U))=0\}.
\]
We then have $X=\{gB\in G/B\mid\Ad^{*}(g)^{-1}\beta\in\mfb{}^{*}\}$,
and we note that $X$ is non-empty thanks to Lemma~\ref{lem:conjugates-cover}
(this will be crucial for the application of Borel's fixed-point theorem
below).

We show that $X$ is closed in $G/B$. Let $M=\{(gB,f)\mid g\in G,\,\Ad^{*}(g)^{-1}f\in\mfb{}^{*}\}$
and $\mathrm{pr}_{2}:M\rightarrow\mfg^{*}$ be as in the proof of
Lemma~\ref{lem:conjugates-cover}. We have 
\[
\mathrm{pr}_{2}^{-1}(\beta)=\{(gB,\beta)\in G/B\times\{\beta\}\mid\Ad^{*}(g)^{-1}\beta\in\mfb^{*}\},
\]
so $\mathrm{pr}_{2}^{-1}(\beta)$ is closed in $M$ since $\beta$
is a closed point. Moreover, we have proved that $M$ is closed in
$G/B\times\mfg^{*}$, so $\mathrm{pr}_{2}^{-1}(\beta)$ is closed
in $G/B\times\mfg^{*}$. To conclude that $X$ is closed in $G/B$,
it remains to note that the map $\lambda:G/B\rightarrow G/B\times\mfg^{*}$
given by $\lambda(gB)=(gB,\beta)$ is a morphism of varieties, and
that
\[
\lambda^{-1}(\mathrm{pr}_{2}^{-1}(\beta))=\{gB\mid\Ad^{*}(g)^{-1}\beta\in\mfb{}^{*}\}=X.
\]

Now, since $X$ is closed in the complete variety $G/B$, it is itself
complete. Any subgroup of $C_{G}(\beta)$ acts on $X$, because for
$gB\in X$ and $h\in C_{G}(\beta)$ we have 
\[
\beta(\Ad(hg)\Lie(U))=(\Ad^{*}(h^{-1})\beta)(\Ad(g)\Lie(U))=\beta(\Ad(g)\Lie(U))=0,
\]
so $hgB\in X$. Thus $B_{\beta}$ acts on $X$ and since it is a connected
solvable group, Borel's fixed-point theorem implies that there exists
a $gB\in X$ such that $hgB=gB$, for all $h\in B_{\beta}$; thus
$B_{\beta}\subseteq gBg^{-1}$. Setting $V=gUg^{-1}$ we then have
$U_{\beta}\subseteq V$, and 
\[
\beta(\Lie(V))=\beta(\Ad(g)\Lie(U))=0.
\]
\end{proof}
We recall that maximal unipotent subgroups of a connected linear
algebraic group over $k$ coincide with unipotent radicals of Borel
subgroups (see \cite[30.4]{humphreys}, where one immediately reduces
to reductive groups by taking unipotent radicals). For an algebraic
group $H$, we let $H^{\circ}$ denote the connected component of
the identity.
\begin{lem}
\label{lem:H_beta}For any $\beta\in\mfg^{*}$ there exists a closed
subgroup $H_{\beta}$ of $C_{G_{2}}(\beta)^{\circ}$ and a maximal
unipotent subgroup $U$ of $G$ such that:
\begin{enumerate}[label=(\roman*)]
\item \label{enu:H_beta-1}$H_{\beta}G^{1}$ is a maximal unipotent subgroup
of $C_{G_{2}}(\beta)^{\circ}$,
\item \label{enu:H_beta-2}$H_{\beta}\cap G^{1}=\exp(\Lie(U))$,
\item \label{enu:H_beta-3}$\beta(\Lie(U))=0$.
\end{enumerate}
\end{lem}
\begin{proof}
Let $U_{\beta}$ be a maximal unipotent subgroup of $C_{G}(\beta)^{\circ}$.
Then $U_{\beta}$ is the unipotent radical of a Borel subgroup of
$C_{G}(\beta)^{\circ}$ (so, in particular, $U_{\beta}$ is connected).
By Lemma~\ref{lem:key-Borel-centr} there exists a Borel subgroup
$B$ of $G$ with unipotent radical $U$ containing $U_{\beta}$,
and such that $\beta(\Lie(U))=0$. Given this $U$, it will therefore
be enough to prove the existence of an $H_{\beta}$ such that \ref{enu:H_beta-1}
and \ref{enu:H_beta-2} hold.

By \cite[XXVI, 3.5]{SGA3} (see also \cite[XXVI,7.15]{SGA3}), there
exists a Borel subgroup scheme $\mathbb{B}$ of $\G$ over $\cO_{2}^{\mathrm{ur}}$
such that $\mathbb{B}\times k=B$. Let $\mathbb{U}$ be the unipotent
radical of $\mathbb{B}$ (see \cite[XXII, 5.11.4\,(ii)]{SGA3} as
well as \cite[5.2.5]{Conrad-Red-grp-sch}), so that $\mathbb{U}\times k=U$.
Let $U_{2}=\cF(\mathbb{U})$ be the Greenberg transform, and define
\[
H_{\beta}=U_{2}\cap C_{G_{2}}(\beta)^{\circ}.
\]
Let $u\in U\cap C_{G}(\beta)^{\circ}$. Since $\mathbb{U}$ is smooth,
there exists an element $\hat{u}\in U_{2}$ such that $\rho(\hat{u})=u$,
and since $C_{G_{2}}(\beta)^{\circ}=\rho^{-1}(C_{G}(\beta)^{\circ})$,
we must have $\hat{u}\in C_{G_{2}}(\beta)^{\circ}$, so that $\hat{u}\in H_{\beta}$.
Thus

\[
\rho(H_{\beta})=\rho(U_{2}\cap C_{G_{2}}(\beta)^{\circ})=U\cap C_{G}(\beta)^{\circ}=U_{\beta}.
\]
Since $G^{1}$ is unipotent, normal in $G$ and $\rho(H_{\beta}G^{1})=U_{\beta}$,
it follows that $H_{\beta}G^{1}$ is a maximal unipotent subgroup
of $C_{G_{2}}(\beta)^{\circ}$, proving \ref{enu:H_beta-1}.

Next, since $C_{G_{2}}(\beta)^{\circ}$ contains $G^{1}$, we have
\[
H_{\beta}\cap G^{1}=U_{2}\cap G^{1}=\Ker(\rho:\mathbb{U}(\cO_{2}^{\mathrm{ur}})\rightarrow\mathbb{U}(k))=\exp(\Lie(U)),
\]
proving \ref{enu:H_beta-2}. Finally, as we have already noted, $\beta(\Lie(U))=0$
holds by our choice of $U$, so \ref{enu:H_beta-3} holds.
\end{proof}

\subsection{Very good primes and component groups of centralisers \label{subsec:Very-good-primes}}

We recall the notions of good and very good primes. If $H$ is a connected
almost simple group over $k$, the prime $p=\chara k$ is \emph{good}
for $H$ if any of the following conditions hold:
\begin{itemize}
\item $H$ is of type $A_{n}$,
\item $H$ is of type $B_{n}$, $C_{n}$ or $D_{n}$ and $p\neq2$,
\item $H$ is of type $G_{2}$, $F_{4}$, $E_{6}$ or $E_{7}$ and $p>3$,
\item $H$ is of type $E_{8}$ and $p>5$,
\end{itemize}
(see, for example, \cite[I, 4]{Springer-Steinberg} or \cite[Definition~2.5.2]{Letellier-book}).
If, moreover, $p$ does not divide $n+1$ whenever $H$ is of type
$A_{n}$, then $p$ is said to be \emph{very good} for $H$ (see,
for example, \cite[Definition~2.5.5]{Letellier-book}). There is also
a notion of torsion prime for $H$ due to Steinberg \cite{Steinberg-Torsion}
(cf.~\cite[Definition~2.5.4]{Letellier-book}). If $p$ is very good
for $H$, then it is not a torsion prime for $H$ (see \cite[Remark~2.5.6]{Letellier-book}).
Now let $G'$ be the derived group of the reductive group $G$. Then
$G'$ is semisimple and $p$ is said to be good/very good/torsion
for $G$ if $p$ is good/very good/torsion for each of the simple
components of $G'$.

If there exists a $G$-equivariant bijection $\mfg\iso\mfg^{*}$,
then each centraliser of an element in $\mfg^{*}$ equals a centraliser
of an element in $\mfg$. Such a bijection exists when there exists
a non-degenerate $G$-invariant bilinear form on $\mfg$, and this
is always the case when $p$ is very good for $G$ (see \cite[Proposition~2.5.12]{Letellier-book}). 

For $\beta\in\mfg^{*}$, let $A(\beta)$ denote the component group
$C_{G}(\beta)/C_{G}(\beta)^{\circ}$. Note that since $G^1$ is connected and normal in $G$, we have
\begin{equation}\label{eq:CG2-CG}
C_{G_{2}}(\beta)/C_{G_{2}}(\beta)^{\circ}\cong C_{G}(\beta)/C_{G}(\beta)^{\circ}=A(\beta).
\end{equation}

\begin{lem}
\label{lem:p-not-div-comp-grp}Assume that $p$ is good and not a
torsion prime for $G$, and that there exists a $G$-equivariant bijection
$\mfg\iso\mfg^{*}$ (e.g., these conditions hold when $p$ is very
good for $G$, or when $G=\GL_{n}$). Then, for any $\beta\in\mfg^{*}$,
$p$ does not divide $|A(\beta)|$.
\end{lem}
\begin{proof}
Using the $G$-equivariant bijection $\mfg^{*}\iso\mfg$, we may replace
$\beta$ by an element $X\in\mfg$. We reduce to centralisers of nilpotent
elements in the standard way: Let $X=X_{s}+X_{n}$ be the Jordan decomposition
of $X$, where $X_{s}$ is semisimple and $X_{n}$ is nilpotent. Uniqueness
of Jordan decomposition implies that $C_{G}(X)=C_{C_{G}(X_{s})}(X_{n})$.
Since $p$ is not torsion for $G$, \cite[Theorem~3.14]{Steinberg-Torsion}
implies that $C_{G}(X_{s})$ is connected. Moreover, by \cite[Proposition~2.6.4]{Letellier-book},
it is reductive. By \cite[Proposition~16]{McNinch-Sommers}, $p$
is good (but not necessarily very good) for $C_{G}(X_{s})$, and $X_{s}\in C_{\mfg}(X_{s})=\Lie(C_{G}(X_{s}))$
(since $C_{G}(X_{s})$ is smooth; see \cite[Proposition~1.10]{Humphreys-Conjcl})
so we are reduced to proving the lemma in the case when $p$ is good
for $G$ and $\beta$ is replaced by a nilpotent element in $\mfg$. 

Assume that $p$ is good for $G$ and let $X\in\mfg$ be nilpotent.
By \cite[Proposition~5]{McNinch-Sommers}, there exists a $G$-equivariant
bijection between the nilpotent variety in $\mfg$ and the unipotent
variety in $G$. Thus $C_{G}(X)=C_{G}(u)$, for some unipotent element
$u\in G$. By \cite[III, 3.15]{Springer-Steinberg} (see \cite[Proposition~12, Corollary~13]{McNinch-Sommers}
for the extension to reductive groups), every element in the component
group $A(u)$ of $C_{G}(u)$ is represented by a semisimple element
in $C_{G}(u)$. By Jordan decomposition, the image of a semisimple
element under a homomorphism of affine algebraic groups is semisimple,
and semisimple elements in a finite group like $A(u)$ are exactly
the $p'$-elements, that is, elements not divisible by $p$. Thus
the group $A(u)$ has no element of order $p$. The lemma follows.
\end{proof}
\begin{rem}
\label{rem:(a)-pretty-good}(a) We do not know whether the converse
of Lemma~\ref{lem:p-not-div-comp-grp} holds. The hypotheses on $p$
in the lemma imply that $p$ is a pretty good prime for $G$ (see
\cite{Herpel-smooth}). Indeed, assume for simplicity that $G$ is
simple with root system $\Phi$ and dual root system $\Phi^{\vee}$
(with respect to some maximal torus). Then $p$ is good for $G$ if
$\Z\Phi/\Z\Phi'$ has no $p$-torsion, for any closed subsystem $\Phi'$
(see \cite[I, 4]{Springer-Steinberg}). Moreover, $p$ is not a torsion
prime for $G$ if $\Z\Phi^{\vee}/\Z\Phi'^{\vee}$ has no $p$-torsion,
for any $\Phi'$, and if $p$ does not divide the order of the fundamental
group $\pi_{1}(G)$ (see \cite[Definition~2.5.4]{Letellier-book}).
By \cite[proof of Lemma~2.12\,(a)]{Herpel-smooth}, the assumption
that $p$ is good and $\Z\Phi^{\vee}/\Z\Phi'^{\vee}$ has no $p$-torsion
for any $\Phi'$ is equivalent to $p$ being pretty good for $G$.
Thus, the hypotheses in Lemma~\ref{lem:p-not-div-comp-grp} are equivalent
to $p$ being pretty good, not dividing the order of $\pi_{1}(G)$
and such that there exists a $G$-equivariant bijection $\mfg\iso\mfg^{*}$.\medskip

(b) In general, many elements $\beta\in\mfg^{*}$ satisfy the conclusion
of Lemma~\ref{lem:p-not-div-comp-grp}, even when some of the hypotheses
of the lemma fail. For example, take any $G$-invariant bilinear (but
not necessarily non-degenerate) form $\langle{}\cdot{},{}\cdot{}\rangle$
on $\mfg$. Then $X\mapsto\langle X,{}\cdot{}\rangle$ defines a $G$-equivariant
map $\mfg\rightarrow\mfg^{*}$, and every element in $\mfg^{*}$ in
the image of this map will satisfy the conclusion of the lemma whenever
$p$ is a good and non-torsion for $G$. For example, when $G=\SL_{n}$
and $\langle{}\cdot{},{}\cdot{}\rangle$ is the trace form, this applies
for any $p$ (in particular, when $p\mid n$).

On the other hand, when $G=\SL_{n}$ and $p\mid n$, the conclusion
of Lemma~\ref{lem:p-not-div-comp-grp} does not hold in general,
even though $p$ is good and non-torsion for $G$. For example, when
$p=n=2$, $\mfg^{*}$ may be identified with $\M_{2}(k)/Z$, where
$\M_{2}(k)$ is the $2\times2$ matrices and $Z$ is the subalgebra
of scalar matrices. It is easy to see that when $\beta=\left(\begin{smallmatrix}0 & 0\\
0 & 1
\end{smallmatrix}\right)$, the component group of $C_{G}(\beta+Z)$ has order~$2$. 
\end{rem}

\section{Representations\label{sec:Representations}}

\subsection{Clifford theory set up\label{subsec:Clifford-theory}}

For a finite group $\Gamma$, we will write $\Irr(\Gamma)$ for the
set of irreducible complex representations of $\Gamma$ (up to isomorphism).
If $\Gamma'\subseteq\Gamma$ is a subgroup and $\rho$ is a representation
of $\Gamma'$, we will write $\Irr(\Gamma\mid\rho)$ for the subset
of $\Irr(\Gamma)$ consisting of representations which have $\rho$
as an irreducible constituent when restricted to $\Gamma$. Recall
the notation introduced in Section~\ref{sec:Alg-grps-Lie}. 

Fix a non-trivial irreducible character $\psi:\F_{q}\rightarrow\C^{\times}$.
For $\beta\in(\mfg^{*})^{F^{*}}$, define the character $\psi_{\beta}\in\Irr((G^{1})^{F})$
by
\[
\psi_{\beta}(\exp(X))=\psi(\langle\beta,X\rangle),\qquad\text{for }X\in\mfg^{F}
\]
Note that here $\langle\beta,X\rangle\in\F_{q}$ and also, $\exp(\mfg^{F})=(G^{1})^{F}$
by Lemma~\ref{lem:exp-compatibility}. Recall the notation $\sigma^{(i)}$
from Section~\ref{subsec:Frobenius-twists}.
\begin{lem}
The function $\beta\mapsto\psi_{\beta}$ defines an isomorphism of
abelian groups $(\mfg^{*})^{F^{*}}\rightarrow\Irr((G^{1})^{F})$ and
for any $g\in G_{2}^{F}$, we have
\[
\Ad^{*}(\sigma^{(i)}(g))\beta\longmapsto\psi_{\beta}^{g},\qquad\text{for }i\in\{0,p\},
\]
where $G_{2}^{F}$ acts on $(\mfg^{*})^{F^{*}}$ via its quotient
$G^{F}$, that is, $\Ad^{*}(g)\beta=\Ad^{*}(\bar{g})\beta$, where
$\bar{g}=\rho(g)$.
\end{lem}
\begin{proof}
The function $\beta\mapsto\psi_{\beta}$ is an additive injective
homomorphism because of the linearity (in the first variable) and
non-degeneracy of the form $\langle{}\cdot{},{}\cdot{}\rangle$, respectively.
Moreover, by Lemma~\ref{lem:exp-compatibility}, $\exp(\Ad(\sigma^{(i)}(\bar{g}))X)=g\exp(X)g^{-1}$,
so for any $X\in\mfg$, we have 
\begin{align*}
\psi_{\Ad^{*}(\sigma^{(i)}(g))\beta}(\exp(X)) & =\psi(\langle\Ad^{*}(\sigma^{(i)}(g))\beta,X\rangle)=\psi(\langle\beta,\Ad^{*}(\sigma^{(i)}(\bar{g}^{-1}))X\rangle)\\
 & =\psi_{\beta}(\exp(\Ad^{*}(\sigma^{(i)}(\bar{g}^{-1}))X))=\psi_{\beta}(g^{-1}\exp(X)g)\\
 & =:\psi_{\beta}^{g}(\exp(X)).
\end{align*}
\end{proof}
Just like $F$, the map $\sigma$ induces an endomorphism $\sigma^{*}$
on $\mfg^{*}$, and it follows immediately from the preceding lemma,
together with the formula $\sigma^{*}(\Ad^{*}(\bar{g})\beta)=\Ad^{*}(\sigma(\bar{g}))\sigma^{*}(\beta)$,
that the stabiliser of $\psi_{\beta}$ in $G_{2}^{F}$ is 
\begin{equation}
\{g\in G_{2}^{F}\mid\Ad^{*}(\sigma^{(i)}(\bar{g}))\beta=\beta\}=\begin{cases}
C_{G_{2}}(\beta)^{F} & \text{if }\tilde{R}=k[t]/t^{2},\\
C_{G_{2}}((\sigma^{*})^{-1}(\beta))^{F} & \text{if }\tilde{R}=W_{2}(k).
\end{cases}\label{eq:stabilisers}
\end{equation}
Here, as elsewhere, we take centralisers with respect to the coadjoint
action (not its Frobenius twist). Recall from Section~\ref{subsec:The-Lie-algebra}
that $\beta\in(\mfg^{*})^{F^{*}}$ implies that $C_{G_{2}}(\beta)$
and $C_{G}(\beta)$ are $F$-stable. The map $\sigma^{*}$ is bijective
and commutes with $F^{*}$, so it restricts to a bijection $\sigma^{*}:(\mfg^{*})^{F^{*}}\rightarrow(\mfg^{*})^{F}$.

The following is an immediate consequence of well known results in
Clifford theory \cite[6.11, 6.17]{Isaacs}:
\begin{lem}
\label{lem:Clifford-theory}Let $\beta\in(\mfg^{*})^{F^{*}}$ and
assume that $\psi_{\beta}$ has an extension $\tilde{\psi}_{\beta}\in\Irr(C_{G_{2}}(\beta)^{F})$.
Then there is a bijection
\begin{align*}
\Irr(C_{G_{2}}(\beta)^{F}/(G^{1})^{F}) & \longrightarrow\Irr(G_{2}^{F}\mid\psi_{\beta})\\
\theta & \longmapsto\pi(\theta):=\Ind_{C_{G_{2}}(\beta)^{F}}^{G_{2}^{F}}(\theta\tilde{\psi}_{\beta}).
\end{align*}
Thus
\[
\#\Irr(G_{2}^{F}\mid\psi_{\beta})=|C_{G_{2}}(\beta)^{F}/(G^{1})^{F}|=|C_{G}(\beta)^{F}|
\]
and 
\[
\dim\pi(\theta)=[G_{2}^{F}:C_{G_{2}}(\beta)^{F}]\cdot\dim\theta=[G^{F}:C_{G}(\beta)^{F}]\cdot\dim\theta.
\]
\end{lem}
In the following, we will prove that an extension of $\psi_{\beta}$
to its stabiliser exists for any $\beta\in(\mfg^{*})^{F^{*}}$, under
suitable hypotheses.

\subsection{Proof of the main theorem\label{subsec:Proof-of-the}}

We will use the following lemma (see \cite[Lemma~4.8]{SS/2017}):
\begin{lem}
\label{lem:p-Sylow-extend}Let $M$ be a finite group, $N$ a normal
$p$-subgroup, and $P$ a Sylow $p$-subgroup of $M$. Suppose that
$\chi\in\Irr(N)$ is stabilised by $M$ and that $\chi$ has an extension
to $P$. Then $\chi$ has an extension to $M$.
\end{lem}
The Sylow $p$-subgroup we will apply the above lemma to is given
by the following result.
\begin{lem}
\label{lem:ident-comp-p-Sylow}Let $\beta$ and $H_{\beta}$ be as
in Lemma~\ref{lem:H_beta}. Let $m\geq1$ be an integer such that
$H_{\beta}$ is $F^{m}$-stable. Then $(H_{\beta}G^{1})^{F^{m}}$
is a Sylow $p$-subgroup of $(C_{G_{2}}(\beta)^{\circ})^{F^{m}}$.
Moreover, if $p$ does not divide $|A(\beta)|$, then $(H_{\beta}G^{1})^{F^{m}}$
is a Sylow $p$-subgroup of $C_{G_{2}}(\beta)^{F^{m}}$.
\end{lem}
\begin{proof}
The first statement follows from Lemma~\ref{lem:H_beta}\,\ref{enu:H_beta-1}
together with \cite[Proposition~3.19\,(i)]{dignemichel}. For the
second statement, note that 
\begin{multline*}
[C_{G_{2}}(\beta)^{F^{m}}:(H_{\beta}G^{1})^{F^{m}}]\\
=[C_{G_{2}}(\beta)^{F^{m}}:(C_{G_{2}}(\beta)^{\circ})^{F^{m}}]\cdot[(C_{G_{2}}(\beta)^{\circ})^{F^{m}}:(H_{\beta}G^{1})^{F^{m}}]
\end{multline*}
and $C_{G_{2}}(\beta)^{F^{m}}/(C_{G_{2}}(\beta)^{\circ})^{F^{m}}\cong(C_{G_{2}}(\beta)/C_{G_{2}}(\beta)^{\circ})^{F^{m}}=A(\beta)^{F^{m}}$
(see \cite[Corollary~3.13]{dignemichel} and \eqref{eq:CG2-CG}), so if $p\nmid|A(\beta)|$,
then $p\nmid[C_{G_{2}}(\beta)^{F^{m}}:(C_{G_{2}}(\beta)^{\circ})^{F^{m}}]$,
hence $p\nmid[C_{G_{2}}(\beta)^{F^{m}}:(H_{\beta}G^{1})^{F^{m}}]$.
\end{proof}
The purpose of the geometric lemmas in Section~\ref{sec:Lemmas-on-algebraic}
is to prove the following result, from which our main theorem immediately
follows.
\begin{prop}
\label{prop:extension-exists} Let $\beta\in(\mfg^{*})^{F^{*}}$ and
assume that $p$ does not divide $|A(\beta)|$. Then the character
$\psi_{\beta}$ has an extension $\tilde{\psi}_{\beta}$ to $C_{G_{2}}(\beta)^{F}$.
\end{prop}
\begin{proof}
Let $H_{\beta}$ and $U$ be as in Lemma~\ref{lem:H_beta}. Then
$H_{\beta}$ (like any algebraic group over $k$) is defined over
some finite extension of $\F_{q}$, or equivalently, it is stable
under some power $F^{m}$, $m\geq1$ of the Frobenius $F$. Since
$G^{1}$ is $F$-stable, the group $H_{\beta}G^{1}$ is $F^{m}$-stable.
By Lemma~\ref{lem:H_beta}\,\ref{enu:H_beta-1}, the group $H_{\beta}G^{1}$
is maximal unipotent in $C_{G_{2}}(\beta)^{\circ}$. Thus, given our
hypothesis on $p$, Lemma~\ref{lem:ident-comp-p-Sylow} implies that
$(H_{\beta}G^{1})^{F^{m}}$ is a Sylow $p$-subgroup of $C_{G_{2}}(\beta)^{F^{m}}$. 

Next, we show that $\psi_{\beta}$ extends to $(H_{\beta}G^{1})^{F^{m}}$.
We claim that 
\[
(H_{\beta}G^{1})^{F^{m}}=H_{\beta}^{F^{m}}(G^{1})^{F^{m}}.
\]
Indeed, the map $\rho:G_{2}\rightarrow G$ is compatible with any
power $F^{m}$ on $G_{2}$ and $G$, respectively, so $\rho$ maps
$(H_{\beta}G^{1})^{F^{m}}$ surjectively onto $U_{\beta}^{F^{m}}$,
where $U_{\beta}$ is the maximal connected unipotent subgroup of
$C_{G_{1}}(\beta)^{\circ}$. Since $G^{1}\cap(H_{\beta}G^{1})^{F^{m}}=(G^{1})^{F^{m}}$,
the kernel is $(G^{1})^{F^{m}}$. Similarly, $\rho$ maps $H_{\beta}^{F^{m}}(G^{1})^{F^{m}}$
surjectively onto $U_{\beta}^{F^{m}}$, with kernel $(G^{1})^{F^{m}}$.
Thus, the finite groups $H_{\beta}^{F^{m}}(G^{1})^{F^{m}}$ and $(H_{\beta}G^{1})^{F^{m}}$
have the same order, so the natural inclusion of the former into the
latter is an isomorphism.

Let $\psi_m$ be an extension of $\psi$ to the additive group of the field of definition of $H_\beta$. The formula $\psi_{\beta,m}(\exp(x))=\psi_m(\langle\beta,x\rangle)$,
for $x\in\mfg^{F^{m}}$ defines an extension $\psi_{\beta,m}$ of
$\psi_{\beta}$ to $(G^{1})^{F^{m}}$. We now show that $\psi_{\beta,m}$
extends to a character of $H_{\beta}^{F^{m}}(G^{1})^{F^{m}}$ which is trivial on $H_{\beta}^{F^{m}}$.
Since $H_{\beta}^{F^{m}}$ is a subgroup of the stabiliser $C_{G_2}(\beta)^{F^m}$ of $\psi_{\beta,m}$, it is easy to see that such an extension exists if $\psi_{\beta,m}$ is trivial on $H_{\beta}^{F^{m}}\cap (G^{1})^{F^{m}}$.
Now,  
$$ H_{\beta}^{F^{m}}\cap(G^{1})^{F^{m}}\subseteq(U_{2}\cap G^{1})^{F^{m}}=\exp(\Lie U)^{F^{m}}$$
(note that $U_{2}$ and $U$ as in the proof of Lemma~\ref{lem:H_beta}
are stable under $F^{m}$ since $H_{_{\beta}}$ and $C_{G_{2}}(\beta)^{\circ}$
are), and hence $$\psi_{\beta,m}(H_{\beta}^{F^{m}}\cap(G^{1})^{F^{m}})\subseteq\psi_m(\langle\beta,(\Lie U)^{F^{m}}\rangle)=\psi_m(\{0\})=\{1\},$$
by Lemma~\ref{lem:H_beta}\,\ref{enu:H_beta-3}. 
Thus $\psi_{\beta,m}$ extends to $H_{\beta}^{F^{m}}(G^{1})^{F^{m}}=(H_{\beta}G^{1})^{F^{m}}$, so by Lemma~\ref{lem:p-Sylow-extend},
$\psi_{\beta,m}$ extends to $C_{G_{2}}(\beta)^{F^{m}}$. Restricting
this (one-dimensional) extension to $C_{G_{2}}(\beta)^{F}$, we obtain the desired extension of $\psi_{\beta}$.
\end{proof}
We can now deduce our main theorem. Given a $\beta\in(\mfg^{*})^{F^{*}}$,
the first assertion of the theorem follows from Lemma~\ref{lem:Clifford-theory}
and Proposition~\ref{prop:extension-exists}, together with the stabiliser
formulas (\ref{eq:stabilisers}). Note that $\beta$ for $\G(W_{2}(\F_{q}))$
is paired up with $\sigma^{*}(\beta)$ for $\G(\F_{q}[t]/t^{2})$.
The second assertion of the theorem follows from the first, together
with Lemma~\ref{lem:p-not-div-comp-grp}. This completes the proof
of Theorem~\ref{thm:A}.

\section{Further directions\label{sec:Further-directions}}
It is natural to ask whether Theorem~\ref{thm:A} remains true for
all $\beta\in(\mfg^{*})^{F^{*}}$ when $p$ is arbitrary. We have
not been able to prove this, but neither do we know a counter-example.
It was stated in \cite[Theorem~1.1]{Pooja-classicalgrps} that for
$p\mid n$ and any integers $n,d\geq1$, one has $\#\Irr_{d}(\SL_{n}(\cO_{2}))=\#\Irr_{d}(\SL_{n}(\cO_{2}')$, with $\cO$ and $\cO'$ as in Section~\ref{sec:Introduction}.
However, the argument given in \cite{Pooja-classicalgrps} for the
crucial Lemma~2.3 has a gap (as acknowledged by the author in private
communication). Namely, it is not clear that $T(\psi_{A})\cap\SL_{n}(\cO_{2})=(Z_{\GL_{n}(\cO_{2})}(s(A))\cap\SL_{n}(\cO_{2}))L(SL)$,
in the notation of \cite{Pooja-classicalgrps}. Theorem~\ref{thm:A}
therefore remains open for $\G=\SL_{n}$, $p\mid n$. 
\begin{question}
\label{prob:1}Let $\G$ be a reductive group scheme over $\Z$. Is
it true that if $p$ is sufficiently large and $r\geq3$, then $\#\Irr_{d}(\G(\cO_{r}))=\#\Irr{}_{d}(\G(\cO_{r}'))$,
for any integer $d\geq1$?
\end{question}
Given Theorem~\ref{thm:A} (i.e., the case $r=2$), this question is equivalent to the question posed in \cite[Section~8.4]{BDOP} (in the case where $\G$ is split, that is, a Chevalley group scheme).

A weaker question is whether the groups $\G(\cO_{r})$ and $\G(\cO'_{r})$
have the same number of conjugacy classes, for sufficiently large
$p$. This was settled in the affirmative in \cite{BDOP}, at least
for Chevalley group schemes (although the bound on $p$ is not explicit).
For $r=2$, no counter-examples are known, even for small primes.
On the other hand, for $r=3$ even this weaker question can fail for small primes, for
it is an exercise to compute that $\SL_{2}(\F_{2}[t]/t^{3})$ has
24 conjugacy classes, while $\SL_{2}(\Z/8)$ has 30 conjugacy classes
(see \cite{MR0444787}) (these numbers can also be verified by computer).

\bibliographystyle{alex}
\bibliography{alex}

\end{document}